\documentclass[12pt]{article}%
\usepackage{graphicx}
\usepackage{amsmath}
\usepackage{a4}
\usepackage{amsfonts}
\usepackage{amssymb}%
\providecommand{\U}[1]{\protect \rule{.1in}{.1in}}
\newtheorem{theorem}{Theorem}

\newtheorem{corollary}[theorem]{Corollary}

\newtheorem{definition}[theorem]{Definition}

\newtheorem{lemma}[theorem]{Lemma}

\newtheorem{proposition}[theorem]{Proposition}
\newtheorem{remark}[theorem]{Remark}

\newenvironment{proof}[1][Proof]{\textbf{#1.} }{\  \rule{0.5em}{0.5em}}

\begin{document}

\title{A Complex of Incompressible Surfaces for handlebodies and the Mapping Class Group.}
\author{Charalampos Charitos$\dagger$, Ioannis Papadoperakis$\dagger$\\and Georgios Tsapogas$\ddagger$\\$\dagger$Agricultural University of Athens \\and $\ddagger$University of the Aegean}
\maketitle

\begin{abstract}
For a genus $g$ handlebody $H_{g}$ a simplicial complex, with vertices being
isotopy classes of certain incompressible surfaces in $H_{g}$, is constructed
and several properties are established. In particular, this complex naturally
contains, as a subcomplex, the complex of curves of the surface $\partial
H_{g}$. As in the classical theory, the group of automorphisms of this complex
is identified with the mapping class group of the handlebody. \footnote{
\textit{{2010 Mathematics Subject Classification:} 57N10, 57N35 }}

\end{abstract}

\section{Definitions and statements of results}

For a compact surface $F,$ the complex of curves $\mathcal{C}\left(  F\right)
,$ introduced by Harvey in \cite{Har}, has vertices the isotopy classes of
essential, non-boundary-parallel simple closed curves in $F.$ A collection of
vertices spans a simplex exactly when any two of them may be represented by
disjoint curves, or equivalently when there is a collection of representatives
for all of them, any two of which are disjoint. Analogously, for a
$3-$manifold $M,$ the disk complex $\mathcal{D}\left(  M\right)  $ is defined
by using the proper isotopy classes of compressing disks for $M$ as the
vertices. It was introduced in \cite{MC}, where it was used in the study of
mapping class groups of $3-$manifolds. In \cite{M-M}, it was shown to be a
quasi-convex subset of $\mathcal{C}\left(  \partial M\right)  .$

By $H_{g}$ we denote the $3-$dimensional handlebody of genus $g\geq2.$ Recall
that a compact connected surface $S\subset H_{g}$ with boundary is properly
embedded if $S\cap \partial H_{g}=\partial S$ and $S$ is transverse to
$\partial H_{g}.$ A \emph{compressing disk} for $S$ is a properly embedded
disk $D$ such that $\partial D$ is essential in $S.$ A properly embedded
surface $S\subset H_{g}$ is \emph{incompressible} if there are no compressing
disks for $S.$ Recall also that a map $F:S\times \left[  0,1\right]
\rightarrow H_{g}$ is a proper isotopy if for all $t\in \left[  0,1\right]  ,$
$F\bigm \vert_{S\times \left \{  t\right \}  }$ is a proper embedding. In this
case we will be saying that $F\left(  S\times \left \{  0\right \}  \right)  $
and $F\left(  S\times \left \{  1\right \}  \right)  $ are properly isotopic in
$H_{g}$ and we will use the symbol $\simeq$ to indicate isotopy in all cases
(curves, surfaces etc).

\begin{definition}
Let $\mathcal{I}\left(  H_{g}\right)  $ be the simplicial complex whose
vertices are the proper isotopy classes of compressing disks for $\partial
H_{g}$ and of properly imbedded boundary- parallel incompressible annuli and
pairs of pants in $H_{g}.$ For a vertex $\left[  S\right]  $ which is not a
class of compressing disks, it is also required that $S$ is isotopic to a
surface $\overline{S}$ embedded in $\partial H_{g}$ via an isotopy
\[
F:S\times \left[  0,1\right]  \rightarrow H_{g}%
\]
with $F\left(  S\times \left \{  0\right \}  \right)  =S,$ $F\left(
S\times \left \{  1\right \}  \right)  =\overline{S}$ and $F$ being proper when
restricted to $\left[  0,1\right)  .$ A collection of vertices spans a simplex
in $\mathcal{I}\left(  H_{g}\right)  $ when any two of them may be represented
by disjoint surfaces in $H_{g}.$
\end{definition}

\noindent Note that the class of properly embedded incompressible surfaces in
$H_{g}$ is very rich. For example, it contains surfaces of arbitrarily high
genus (see \cite{Qiu}, \cite{CPT}) which are not included as vertices in the
complex $\mathcal{I}\left(  H_{g}\right)  $ defined above. Also observe that
there exist properly embedded annuli and pairs of pants which are not isotopic
to a surface entirely contained in $\partial H_{g}.$ The isotopy classes of
such surfaces are also excluded from the vertex set of $\mathcal{I}\left(
H_{g}\right)  .$

Note that we may regard $\mathcal{D}\left(  H_{g}\right)  $ as a subcomplex of
$\mathcal{I}\left(  H_{g}\right)  $ or, by taking boundaries of the
representative disks, of $\mathcal{C}\left(  \partial H_{g}\right)  .$ Note
also that the vertices of $\mathcal{I}\left(  H_{g}\right)  $ represented by
annuli correspond exactly to the vertices of $\mathcal{C}\left(  \partial
H_{g}\right)  $ represented by curves that are essential in $\partial H_{g}$
but are not meridian boundaries. We define the complex of annuli
$\mathcal{A}\left(  H_{g}\right)  $ to be the subcomplex of $\mathcal{I}%
\left(  H_{g}\right)  $ spanned by these vertices. Together, the vertices of
$\mathcal{D}\left(  H_{g}\right)  \cup \mathcal{A}\left(  H_{g}\right)  $ span
a copy of $\mathcal{C}\left(  \partial H_{g}\right)  $ in $\mathcal{I}\left(
H_{g}\right)  ,$ and we regard $\mathcal{C}\left(  \partial H_{g}\right)  $ as
a subcomplex of $\mathcal{I}\left(  H_{g}\right)  .$

Our goal is to show that for a handlebody $H_{g}$ of genus $g\geq2$ the
automorphisms of the complex $\mathcal{I}\left(  H_{g}\right)  $ are all
geometric, that is, they are induced by homeomorphisms of $H_{g}.$ This can be
rephrased by saying that the map
\[
A:\mathcal{MCG}\left(  H_{g}\right)  \rightarrow Aut\left(  \mathcal{I}\left(
H_{g}\right)  \right)
\]
is an onto map, where $Aut\left(  \mathcal{I}\left(  H_{g}\right)  \right)  $
is the group of automorphisms of the complex $\mathcal{I}\left(  H_{g}\right)
$ and $\mathcal{MCG}\left(  H_{g}\right)  $ is the (extended) mapping class
group of $H_{g},$ i.e. the group of isotopy classes of self-homeomorphisms of
$H_{g}.$ Moreover, we will show (see Theorem \ref{main_theorem} below) that
the map $A$ is 1-1 except when $H_{g}$ is the handlebody of genus $2$ in which
case a $\mathbb{Z}_{2}$ kernel is present generated by the hyper-elliptic involution.

\begin{figure}[ptb]
\begin{center}
\includegraphics[scale=1.03]
{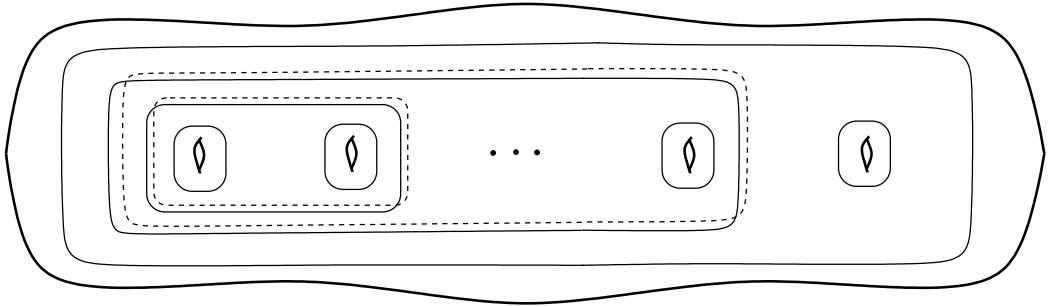}
\end{center}
\caption{Pants decomposition for $H_{g}$ consisting of non-separating,
non-meridian curves, $g\geq3.$}%
\label{non_meridian_pants}%
\end{figure}

For the proof of this result we perform a close examination of links of
vertices in $\mathcal{I}\left(  H_{g}\right)  .$ This examination establishes
that an automorphism $f$ of $\mathcal{I}\left(  H_{g}\right)  $ must map each
vertex $v$ in $\mathcal{I}\left(  H_{g}\right)  $ to a vertex $f\left(
v\right)  $ consisting of surfaces of the same topological type as those in
$v.$ In particular, $f$ induces an automorphism of the subcomplex
$\mathcal{C}\left(  \partial H_{g}\right)  $ which permits the use of the
corresponding result for surfaces (see \cite{I3}, \cite{Luo}).

It is a well known result that for genus $\geq2$ the complex of curves
$\mathcal{C}\left(  \partial H_{g}\right)  $ is a $\delta-$hyperbolic metric
space in the sense of Gromov (see \cite{M-M2},\cite{B}). In the last section
we deduce that the complex $\mathcal{I}\left(  M\right)  $ is itself a
$\delta-$hyperbolic metric space in the sense of Gromov. Moreover, it follows
that $Aut\left(  \mathcal{I}\left(  H_{g}\right)  \right)  $ does not contain
parabolic elements and the hyperbolic isometries of $\mathcal{I}\left(
M\right)  $ correspond to pseudo-Anosov elements of $\mathcal{MCG}\left(
H_{g}\right)  .$

In a recent preprint of M. Korkmaz and S. Schleimer (see \cite{K-S}), it was
shown, in a more general context, that $\mathcal{MCG}\left(  H_{g}\right)  $
and $Aut\left(  \mathcal{D}\left(  H_{g}\right)  \right)  $ are isomorphic.
Apart from this isomorphism, our motivation for constructing the copmplex
$\mathcal{I}\left(  H_{g}\right)  $ is the study of the mapping class group of
a Heegaard splitting in a $3-$manifold $M.$ This group (originally defined for
$\mathbb{S}^{3}$ and often called the Goeritz mapping class group) consists of
the isotopy classes of orientation preserving homeomorphisms of $M$ that
preserve the Heegaard splitting. The mapping class group of a Heegaard
splitting is known to be finitely presented (see \cite{Akb}, \cite{Cho},
\cite{Sch}) only for $M=\mathbb{S}^{3}$ and for a genus $2$ Heegaard
splitting. We aim to examine the corresponding open questions for
$M=\mathbb{S}^{3}$ and Heegaard splittings of genus $\geq3$ as well as for
certain classes of hyperbolic 3-manifolds. For these purposes, the complex
$\mathcal{I}\left(  H_{g}\right)  $ is a suitable building block for defining
a complex encoding the complexity of the Goeritz mapping class group, because
$\mathcal{I}\left(  H_{g}\right)  $ contains a copy of the curve complex of
the boundary surface $\partial H_{g}.$

\subsection{Notation and terminology}

A $3-$dimensional handlebody $H_{g}$ of genus $g$ can be represented as the
union of a handle of index $0$ (i.e. a $3-$ball) with $g$ handles of index $1$
(i.e. $g$ copies of $D^{2}\times \left[  0,1\right]  $).

For an essential simple closed curve $\alpha$ in $\partial H_{g}$ we will be
writing $\left[  \alpha \right]  $ for its isotopy class and the corresponding
vertex in $\mathcal{C}\left(  \partial H_{g}\right)  .$ We will be writing
$\left[  S_{\alpha}\right]  $ for the corresponding vertex in $\mathcal{A}%
\left(  H_{g}\right)  $ where $S_{\alpha}$ is the annulus corresponding to the
curve $\alpha,$ provided that $\alpha$ is not a meridian boundary. We will be
saying that $\left[  S_{\alpha}\right]  $ is an \emph{annular vertex}. If
$\alpha$ is a meridian boundary we will be writing $\left[  D_{\alpha}\right]
$ for the corresponding vertex in $\mathcal{D}\left(  H_{g}\right)  .$ We will
be saying that $\left[  D_{\alpha}\right]  $ is a \emph{meridian vertex} and
$\alpha$ a meridian curve. A vertex in $\mathcal{I}\left(  H_{g}\right)
\setminus \left(  \mathcal{D}\left(  H_{g}\right)  \cup \mathcal{A}\left(
H_{g}\right)  \right)  $ will be called a \emph{pants vertex}.

\begin{figure}[ptb]
\begin{center}
\includegraphics[scale=0.95]
{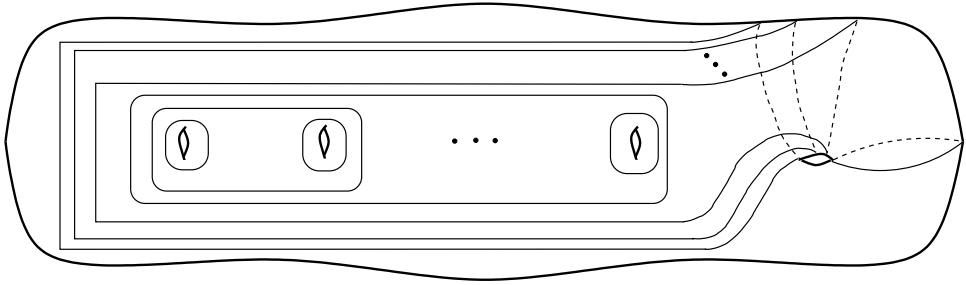}
\end{center}
\caption{Pants decomposition for $H_{g}$ consisting of a single non-separating
meridian curve, and $3g-4$ non-meridian curves, $g\geq3.$}%
\label{meridian_pants}%
\end{figure}\noindent

By writing $\left[  \alpha \right]  \cap \left[  \beta \right]  =\varnothing$ for
non-isotopic curves $\alpha,\beta$ we mean that there exist curves
$\alpha^{\prime}\in \left[  \alpha \right]  $ and $\beta^{\prime}\in \left[
\beta \right]  $ such that $\alpha^{\prime}\cap \beta^{\prime}=\varnothing.$ By
writing $\left[  \alpha \right]  \cap \left[  \beta \right]  \neq \varnothing$ we
mean that for any $\alpha^{\prime}\in \left[  \alpha \right]  $ and
$\beta^{\prime}\in \left[  \beta \right]  ,$ $\alpha^{\prime}\cap \beta^{\prime
}\neq \varnothing.$ By saying that the class $\left[  \alpha \right]  $
intersects the class $\left[  \beta \right]  $ at one point we mean that, in
addition to $\left[  \alpha \right]  \cap \left[  \beta \right]  \neq
\varnothing,$ there exist curves $\alpha^{\prime}\in \left[  a\right]  $ and
$\beta^{\prime}\in \left[  \beta \right]  $ which intersect at exactly one point.

The above notation with square brackets will be similarly used for surfaces.
If $S$ is an incompressible surface we will denote by $Lk\left(  \left[
S\right]  \right)  $ the link of the vertex $\left[  S\right]  $ in
$\mathcal{I}\left(  H_{g}\right)  ,$ namely, for each simplex $\sigma$
containing $\left[  S\right]  $ consider the faces of $\sigma$ not containing
$\left[  S\right]  $ and take the union over all such $\sigma.$ We will use
the notation $\ncong$ to declare that two links are not isomorphic as complexes.

We will also use the classical notation $\Sigma_{n,b}$ to denote a surface of
genus $n$ with $b$ boundary components.

\section{Properties of the complex $\mathcal{I}\left(  H_{g}\right)  $
\label{invariance}}

\bigskip In this section we will show that every automorphism of
$\mathcal{I}\left(  M\right)  $ must preserve the subcomplexes $\mathcal{A}%
\left(  H_{g}\right)  $ and $\mathcal{D}\left(  H_{g}\right)  .$ In
particular, we will show that for $\left[  S\right]  \in \mathcal{I}\left(
H_{g}\right)  ,$ the topological type of the surface $S$ determines the link
of $\left[  S\right]  $ in $\mathcal{I}\left(  H_{g}\right)  $ and vice-versa.
To do this we will find topological properties for the link of each
topological type of surfaces (meridians, annuli and pairs of pants) that
distinguish their links.

It is well known that a pants decomposition for $\partial H_{g}$ is a
collection $\alpha_{1},\ldots,\alpha_{3g-3}$ of $3g-3$ essential,
non-parallel, simple closed curves such that the closure of each component of
the complement of these curves is a pair of pants. The number of pairs of
pants is $2g-2.$ Thus, the maximal number of vertices in a simplex of
$\mathcal{I}\left(  H_{g}\right)  $ is $5g-5.$ In other words the dimension of
$\mathcal{I}\left(  H_{g}\right)  $ is $\leq5g-6.$ To see that simplices of
dimension $5g-6$ actually exist, observe that there exists a pants
decomposition $\alpha_{1},\ldots,\alpha_{3g-3}$ so that each $\alpha_{i}$ is a
non-separating, non-meridian curve for all $i.$ This is displayed in Figure
\ref{non_meridian_pants} for $g\geq3$ and for $g=2$ see Remark \ref{r5} below.
For such a choice of $\alpha_{i}$'s, all $2g-2$ pairs of pants formed by
$\alpha_{1},\ldots,\alpha_{3g-3}$ are incompressible surfaces. Apparently, all
such pairs of pants give rise to distinct elements in $\mathcal{I}\left(
H_{g}\right)  .$ Thus, a pants decomposition $\alpha_{1},\ldots,\alpha_{3g-3}$
with all $\alpha_{i}$'s being non-meridian curves gives rise to $3g-3$ annular
surfaces $S_{\alpha_{1}},\ldots,S_{\alpha_{3g-3}}.$ These surfaces along with
the $2g-2$ pairs of pants formed by $\alpha_{1},\ldots,\alpha_{3g-3}$ give
rise to a simplex in $\mathcal{I}\left(  H_{g}\right)  $ containing $5g-5$
vertices. We have established the following

\begin{proposition}
The dimension of the complex $\mathcal{I}\left(  H_{g}\right)  $ is $5g-6.$
\end{proposition}

We next examine the dimension of $Lk\left(  \left[  D\right]  \right)  $ when
$D$ is a meridian and of $Lk\left(  \left[  S_{\alpha}\right]  \right)  $ when
$S_{\alpha}$ is an annular surface. \begin{figure}[ptb]
\begin{center}
\includegraphics[scale=1.03]{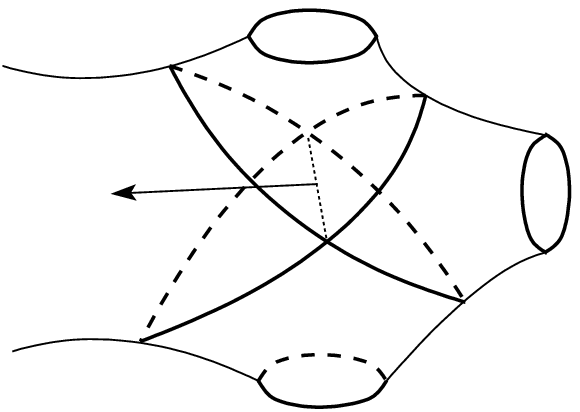}
\end{center}
\par
\begin{picture}(22,12)
\put(102,88){$D_{\beta_{i}} \cap D_{\beta^{\prime}_i}$}
\put(152,35){$\beta_{i}^{\prime}$}
\put(160,137){$\beta_i$}
\put(233,27){$\alpha_3$}
\put(293,88){$\alpha_1$}
\put(233,140){$\alpha_2$}
\end{picture}
\caption{$\,$}%
\label{fig02}%
\end{figure}

\begin{lemma}
\label{dim1}If $S_{\alpha}$ is an annular (incompressible) surface then the
link of the vertex $\left[  S_{\alpha}\right]  $ in $\mathcal{I}\left(
H_{g}\right)  $ has dimension $5g-7.$
\end{lemma}

\begin{proof}
We first assume that $\alpha$ is a separating curve. Then $\alpha$ decomposes
$\partial H_{g}$ into surfaces $\Sigma_{n,1}$ and $\Sigma_{m,1}$ with $m+n=g$
and $m,n\geq1$ with $\alpha$ being isotopic to the boundary of $\Sigma_{n,1}$
as well as to the boundary of $\Sigma_{m,1}.$ To complete the proof in this
case, it suffices to find a pants decomposition for $\partial H_{g}$
consisting of non-meridian curves and containing the curve $\alpha.$ For the
latter, it suffices to show the following

\begin{description}
\item[Claim] $\Sigma_{n,1}$ can be decomposed into $2n-1$ pairs of pants so
that the boundary curves of each are non-meridian when viewed as curves in
$\partial H_{g}.$
\end{description}

The first step is to find pair-wise disjoint non-separating curves $\alpha
_{1},,\ldots,\alpha_{n}$ in $\Sigma_{n,1}$ such that $\alpha_{i}$ does not
bound a disk in $H_{g}\ $for all $i.$ To see this, let $\alpha_{1},\alpha
_{1}^{\prime}$ be two simple non-separating curves in $\partial H_{g}$ such
that the curves $\alpha,\alpha_{1},\alpha_{1}^{\prime}$ bound a pair of pants
in $\partial H_{g}.$ As $\alpha$ is not the boundary of a meridian in $H_{g},$
it is clear that $\alpha_{1},\alpha_{1}^{\prime}$ cannot both be meridian
boundaries in $H_{g}.$ Assuming $\alpha_{1}$ is not meridian boundary, we may
cut $\Sigma_{n,1}$ along $\alpha_{1}$ to obtain a surface $\Sigma_{n-1,3}.$ By
the same argument, we may find a non-separating curve $\alpha_{i}$ in
$\Sigma_{n-(i-1), 2i-1},$ $i=2,\ldots,n$ which is not meridian boundary.

Apparently, cutting $\Sigma_{n,1}$ along $\alpha_{1},,\ldots,\alpha_{n}$ we
obtain a sphere $\Sigma_{0,1+2n}$ with $1+2n$ holes, such that the boundary
components of $\Sigma_{0,1+2n}$ do not bound disks when viewed as curves in
$\partial H_{g}.$ We now claim that we may find pair-wise disjoint curves
$\beta_{1},\ldots,\beta_{2n-2}$ such that $\beta_{j}$ does not bound a disk in
$H_{g}\ $for all $j=1,\ldots,2n-2.$ To see this, let $\beta_{1},\beta
_{1}^{\prime}$ be two simple closed curves in $\Sigma_{0,1+2n}$ such that the
curves $\alpha_{1},\alpha_{2},\beta_{1}$ bound a pair of pants and the curves
$\alpha_{1},\alpha_{3},\beta_{1}^{\prime}$ bound a pair of pants as shown in
Figure 3. If both $\beta_{1},\beta_{1}^{\prime}$ bound properly embedded disks
in $H_{g},$ say $D_{\beta_{1}},D_{\beta_{1}^{\prime}}$ respectively, then
$D_{\beta_{1}}\cap D_{\beta_{1}^{\prime}}$ is a properly embedded arc in
$H_{g}$ which separates $D_{\beta_{1}}$ into two half-disks. Similarly for
$D_{\beta_{1}^{\prime}}.$ Appropriate unions of these half-disks along
$D_{\beta_{1}}\cap D_{\beta_{1}^{\prime}}$ establish a contradiction since
none of $\alpha_{1},\alpha_{2},\alpha_{3}$ is a meridian boundary. Thus, at
least one of $\beta_{1},\beta_{1}^{\prime},$ say $\beta_{1},$ does not bound a
disk. Cutting $\Sigma_{0,1+2n}$ along $\beta_{1}$ we obtain a pair of pants
and a surface $\Sigma_{0,1+2n-1}$ which has the same property as
$\Sigma_{0,1+2n},$ namely, all boundary components of $\Sigma_{0,1+2n-1}$ do
not bound disks when viewed as curves in $\partial H_{g}.$ By applying the
same argument repeatedly, we may find the desired collection of curves
$\beta_{1},\ldots,\beta_{2n-2}$ none of which is a meridian boundary.
Apparently, the collection of curves $\beta_{1},\ldots,\beta_{2n-2}$
decomposes $\Sigma_{0,1+2n}$ into $2n-1$ pairs of pants as required. This
completes the proof of the Claim and the proof of the lemma in the case
$\alpha$ is separating.

Assume now that $\alpha$ is non-separating. Using two copies of $\alpha$ and a
simple arc joining them we may construct a separating curve $\beta$ which
decomposes $\partial H_{g}$ into surfaces $\Sigma_{g-1,1}$ and $\Sigma_{1,1}$
with $\beta$ being isotopic to the boundary of $\Sigma_{g-1,1}$ as well as to
the boundary of $\Sigma_{1,1}.$ Note that $\Sigma_{1,1}$ contains $\alpha.$
Then by the above claim we have that $\Sigma_{g-1,1}$ can be decomposed into
$2\left(  g-1\right)  -1$ (incompressible) pairs of pants by using
non-meridian curves $\alpha_{i},$ $i=1,\ldots,3g-5$ contained in
$\Sigma_{g-1,1}$ together with the cirve $\beta.$ By adding the curve $\alpha$
we obtain a pants decomposition $\alpha_{1},\ldots,\alpha_{3g-5},\beta,\alpha$
with all curves being non-meridian. Hence, $\left[  S_{\alpha}\right]  $ is
contained in a simplex of maximum dimension, namely, of dimension $5g-6$ which
shows that the dimension of $Lk\left(  \left[  S_{\alpha}\right]  \right)  $
is $5g-7.$
\end{proof}

\begin{lemma}
\label{dim2}If $D$ is a meridian then the link of the vertex $\left[
D\right]  $ in $\mathcal{I}\left(  H_{g}\right)  $ has dimension $5g-9.$
\end{lemma}

\begin{proof}
First assume that $\left[  D\right]  $ is non-separating. We may find a pants
decomposition $\alpha_{1},\ldots,\alpha_{3g-4},\alpha_{3g-3}=\partial D$ for
$\partial H_{g}$ such that $\alpha_{i}$ is non-meridian for all $i=1,\ldots
,3g-4$ (see Figure \ref{meridian_pants}). This collection of curves decomposes
$\partial H_{g}$ into $2g-2$ pairs of pants such that exactly two of these
have $\partial D$ as boundary component and, hence, they are compressible
surfaces. Thus, a non-separating meridian $\left[  D\right]  $ is contained in
a simplex with $3g-3+2g-4$ vertices and, hence, the dimension of $Lk\left(
\left[  D\right]  \right)  $ is $\geq5g-9.$ Let now $\alpha_{1}^{\prime
},\ldots,\alpha_{3g-4}^{\prime},\alpha_{3g-3}=\partial D$ be any pants
decomposition with corresponding pairs of pants $P_{1},\ldots,P_{2g-2}$ such
that one of them, say $P_{1},$ has two boundary components isotopic to
$\partial D.$ Then the third boundary component of $P_{1}$ will also be a
meridian, thus, another pair of pants distinct from $P_{1}$ will also be
compressible. This shows that a class $\left[  D\right]  $ with $D$
non-separating meridian cannot be contained in a simplex of more that $5g-7$
vertices and, thus, $Lk\left(  \left[  D\right]  \right)  $ is equal to
$5g-9.$

If $\left[  D\right]  $ is separating, it is clear that any decomposition
$\alpha_{1},\ldots,\alpha_{3g-4},\alpha_{3g-3}=\partial D$ for $\partial
H_{g}$ with $\alpha_{i}$ being non-meridian for all $i=1,\ldots,3g-4$ has the
property that exactly two of the corresponding pairs of pants are compressible
and we work similarly.
\end{proof}

\begin{proposition}
\label{non_iso}Let $\left[  D\right]  $ be a meridian vertex, $\left[
S_{\alpha}\right]  $ an annular vertex and $\left[  P\right]  $ a pants
vertex. Then the links $Lk\left(  \left[  D\right]  \right)  ,Lk\left(
\left[  S_{\alpha}\right]  \right)  $ and $Lk\left(  \left[  P\right]
\right)  $ are pair-wise non-isomorphic as complexes.
\end{proposition}

\begin{proof}
By the previous two Lemmata, the links of the vertices $\left[  D\right]  $
and $\left[  S_{\alpha}\right]  $ have distinct dimensions, hence, it is clear
that $Lk\left(  \left[  D\right]  \right)  \ncong Lk\left(  \left[  S_{\alpha
}\right]  \right)  .$ It remains to distinguish $Lk\left(  \left[  P\right]
\right)  $ from $Lk\left(  \left[  D\right]  \right)  $ and $Lk\left(  \left[
S_{\alpha}\right]  \right)  .$

Let $\left[  P\right]  $ be a vertex in $\mathcal{I}\left(  M\right)  $ such
that $P$ is a pair of pants with boundary components $\beta,\gamma,\delta.$
The vertices in $Lk\left(  \left[  P\right]  \right)  $ form a cone graph,
that is, the vertex $\left[  S_{\beta}\right]  $ belongs to $Lk\left(  \left[
P\right]  \right)  $ and is connected by an edge with any vertex in $Lk\left(
\left[  P\right]  \right)  .$ We will reach a contradiction by showing that
\begin{equation}
\forall \left[  Q\right]  \in Lk\left(  \left[  D\right]  \right)
,\exists \left[  R\right]  \in Lk\left(  \left[  D\right]  \right)  :\left[
Q\right]  \cap \left[  R\right]  \neq \varnothing \tag{$\ast$}\label{starp}%
\end{equation}
and similarly for $Lk\left(  \left[  S_{\alpha}\right]  \right)  .$ For, if
$\beta_{Q}$ is a boundary component of a surface representing $\left[
Q\right]  \in Lk\left(  \left[  D\right]  \right)  $ then there exists a curve
$\gamma$ such that $\partial D\cap \gamma=\varnothing$ and $\gamma \cap \beta
_{Q}\neq \varnothing.$ Let $\left[  R\right]  $ be the vertex represented by
$S_{\gamma}$ if $\gamma$ is non-meridian and by $D_{\gamma}$ if $\gamma$ is a
meridian boundary. Then $\left[  R\right]  \in Lk\left(  \left[  D\right]
\right)  $ is the required vertex which is not connected by an edge with
$\left[  Q\right]  ,$ thus $Lk\left(  \left[  D\right]  \right)  $ satisfies
property $(\ast).$ Similarly, we show that $Lk\left(  \left[  S_{\alpha
}\right]  \right)  $ also satisfies property $(\ast).$
\end{proof}

\begin{remark}
\label{r5}Let $\alpha,\beta,\gamma$ be non-separating curves in $\partial
H_{2}$ decomposing $\partial H_{2}$ into two components which we denote by
$P,P^{\prime}.$ Note that $P,P^{\prime}$ may not be isotopic. To see this,
denote by $f_{1},f_{2}$ the generators of $\pi_{1}\left(  H_{2}\right)  $
corresponding to the longitudes of $H_{2}.$ We may choose non-separating
curves $\alpha,\beta$ on $\partial H_{2}$ which represnt the second powers
$f_{1}^{2},f_{2}^{2}$ up to conjugacy. Choose an essential non-separating
curve $\gamma$ such that $\alpha,\beta,\gamma$ are mutually disjoint and non
isotopic. These curves separate $\partial H_{2}\ $into two components (pairs
of pants) $P$ and $P^{\prime}.$ If $P,P^{\prime}$ were isotopic then $H_{2}$
would be homeomorphic to the product $P\times \left[  0,1\right]  $ and any two
of the boundary components of $P$ would give rise to generators for $\pi
_{1}\left(  H_{2}\right)  .$ Since neither $\alpha \simeq f_{1}^{2}$ nor
$\beta \simeq f_{2}^{2}$ are generators for the free group on $f_{1},f_{2}$ it
follows that, for this particular choice of $\alpha,\beta,\gamma,$ the
surfaces $P,P^{\prime}$ are not isotopic.\smallskip
\end{remark}

\section{Proof of the Main Theorem\label{main}}

\bigskip Let
\[
A:\mathcal{MCG}\left(  H_{g}\right)  \rightarrow Aut\left(  \mathcal{I}\left(
H_{g}\right)  \right)
\]
be the map sending a mapping class $F$ to the automorphism it induces on
$\mathcal{I}\left(  H_{g}\right)  ,$ that is, $A\left(  F\right)  $ is given
by
\[
A\left(  F\right)  \left[  S\right]  :=\left[  F\left(  S\right)  \right]  .
\]

\begin{theorem}
\label{main_theorem} The map $A:\mathcal{MCG}\left(  H_{g}\right)  \rightarrow
Aut\left(  \mathcal{I}\left(  H_{g}\right)  \right)  $ is onto for $g\geq2$
and injective for $g\geq3.$ For $g=2,$ $A$ has a $\mathbb{Z}_{2}-$kernel
generated by the hyper-elliptic involution.
\end{theorem}

We will use the following immediate Corollary of Proposition \ref{non_iso}.

\begin{corollary}
\label{non_iso_more}Automorphisms of $\mathcal{I}\left(  H_{g}\right)  $
preserve all types (meridian, annular and pants) of vertices.
\end{corollary}

We will also need the following

\begin{lemma}
\label{l1}If $f$ $\in Aut\left(  \mathcal{I}\left(  H_{g}\right)  \right)  $
and $f|_{\mathcal{D}\left(  H_{g}\right)  \cup \mathcal{A}\left(  H_{g}\right)
}=id_{\mathcal{D}\left(  H_{g}\right)  \cup \mathcal{A}\left(  H_{g}\right)  }$
then $f\left(  \left[  S\right]  \right)  =\left[  S\right]  $ for any vertex
$\left[  S\right]  \in \mathcal{I}\left(  M\right)  $ except in the case
mentioned in Remark \ref{r5}, namely, if $g=2$ and $P$ is a pair of pants with
all boundary components of $\partial P$ being separating curves decomposing
$\partial H_{2}$ into $2$ components $P,P^{\prime}$, then either, $f\left(
\left[  P\right]  \right)  =\left[  P\right]  $ or, $f\left(  \left[
P\right]  \right)  =\left[  P^{\prime}\right]  .$
\end{lemma}

\begin{proof}
We have to show that $f\in Aut\left(  \mathcal{I}\left(  H_{g}\right)
\right)  $ fixes every vertex $\left[  P\right]  $ where $P$ is a pair of
pants. Let $\left[  P\right]  $ be such a vertex in $\mathcal{I}\left(
H_{g}\right)  .$ By Corollary \ref{non_iso_more} it is clear that $f\left(
\left[  P\right]  \right)  $ is a vertex $\left[  P^{\prime}\right]  $ with
$P^{\prime}$ being a pair of pants. Denote by $\alpha_{1},$ $\alpha_{2},$
$\alpha_{3}$ the boundary components of $P$ and, similarly, $\alpha
_{1}^{\prime},$ $\alpha_{2}^{\prime},$ $\alpha_{3}^{\prime}$ for $P^{\prime}.$
If $\left[  \alpha_{i_{0}}\right]  \cap \left[  \alpha_{j_{0}}^{\prime}\right]
\neq \varnothing$ for some $i_{0},j_{0}\in \{1,2,3\}$ then the vertex $\left[
S_{\alpha_{i_{0}}}\right]  $ is connected by an edge with $\left[  P\right]  $
and is not connected by an edge with $\left[  P^{\prime}\right]  .$ As
$\left[  S_{\alpha_{i_{0}}}\right]  $ is fixed by $f,$ it follows that
$f\left(  \left[  P\right]  \right)  $ cannot be equal to $\left[  P^{\prime
}\right]  .$ Thus, we may assume that
\begin{equation}
\left[  \alpha_{i}\right]  \cap \left[  \alpha_{j}^{\prime}\right]
=\varnothing \mathrm{\ for\ all\ }i,j=1,2,3. \tag{$\ast \ast$}%
\end{equation}
Consider the following property:
\begin{equation}
\mathrm{Up\ to\ change\ of\ enumeration,\ }\alpha_{i}\simeq \alpha_{i}^{\prime
}\mathrm{\ for\ }i=1,2,3. \tag{$\ast \ast \ast$}%
\end{equation}

If property $(\ast \ast \ast)$ holds then $P\simeq P^{\prime}$ unless $g=2$ and
$\alpha_{1},$ $\alpha_{2},$ $\alpha_{3}$ are all non-separating curves which
decompose $\partial H_{2}$ into $2$ pairs of pants (cf. Remark \ref{r5}) which
may or may not be isotopic. Thus, if property $(\ast \ast \ast)$ holds then
either $f\left(  \left[  P\right]  \right)  =\left[  P\right]  $ or the
exception in the statement of the lemma occurs.

We examine now the case where $g\geq3$ and property $(\ast \ast \ast)$ does not
hold. By assumption $(\ast \ast),$ we may cut $\partial H_{g}$ along
$\alpha_{1},$ $\alpha_{2},$ $\alpha_{3}$ to obtain either

\begin{itemize}
\item the surface $P$ and a surface $\Sigma_{g-2,3}$ (if all $\alpha_{1},$
$\alpha_{2},$ $\alpha_{3}$ are non-separating) or,

\item the surface $P,$ a surface $\Sigma_{g_{1},1}$ and a surface
$\Sigma_{g-g_{1}-1,2}$ for some $0<g_{1}<g$ (if exactly one of $\alpha_{1},$
$\alpha_{2},$ $\alpha_{3}$ is separating and the other two curves are
non-isotopic) or,

\item the surface $P$ and a surface $\Sigma_{g-1,1}$ (if exactly one of
$\alpha_{1},$ $\alpha_{2},$ $\alpha_{3}$ is separating and the other two
curves are isotopic) or,

\item the surface $P$ and surfaces $\Sigma_{g_{1},1},\Sigma_{g_{2},1}%
,\Sigma_{g_{3},1}$ for some $g_{1},g_{2},g_{3}\geq1$ with $g_{1}+g_{2}%
+g_{3}=g$ (if all $\alpha_{1},$ $\alpha_{2},$ $\alpha_{3}$ are separating)
\end{itemize}

\noindent Note that if $P$ is a pair of pants, it is impossible to have
exactly two of its boundary curves $\alpha_{1},$ $\alpha_{2},$ $\alpha_{3}$
being separating. In all cases, $P^{\prime}$ is contained in a surface of the
form $\Sigma_{g^{\prime},b}$ for some $g^{\prime}\in \left \{  1,\ldots
,g-1\right \}  $ and $b\in \left \{  1,2,3\right \}  $ mentioned above. Thus, we
may find a non-meridian curve $\alpha$ in $\partial H_{g}$ such that
\[
\alpha \cap \alpha_{i}=\varnothing,\forall i=1,2,3\mathrm{\ and\ }\left[
\alpha \right]  \cap \left[  \alpha_{j_{0}}^{\prime}\right]  \neq \varnothing
\mathrm{\ for\ some\ }j_{0}\in \{1,2,3\}.
\]
Then, for the annular surface $S_{\alpha}$ we have that $\left[  S_{\alpha
}\right]  $ is connected by an edge with $\left[  P\right]  $ and is not
connected by an edge with $\left[  P^{\prime}\right]  .$ As $\left[
S_{\alpha}\right]  $ is fixed by $f,$ it follows that $f\left(  \left[
P\right]  \right)  $ cannot be equal to $\left[  P^{\prime}\right]  .$ This
completes the proof of the lemma.
\end{proof}


\begin{proof}
[Proof of Theorem \ref{main_theorem}]We will use the corresponding result for
surfaces, see \cite{I3},\cite{Luo}, which applies to the boundary of the
handlebody $\partial H_{g}.$

We first show that every $f\in Aut\left(  \mathcal{I}\left(  H_{g}\right)
\right)  $ is geometric. By Proposition \ref{non_iso} we know that $f\left(
\mathcal{A}\left(  H_{g}\right)  \right)  =\mathcal{A}\left(  H_{g}\right)  $
and $f\left(  \mathcal{D}\left(  H_{g}\right)  \right)  =\mathcal{D}\left(
H_{g}\right)  .$ In particular, $f\left(  \mathcal{C}\left(  \partial
H_{g}\right)  \right)  =\mathcal{C}\left(  \partial H_{g}\right)  .$ The
restriction $f|_{\mathcal{C}\left(  \partial H_{g}\right)  }$ of $f$ on
$\mathcal{C}\left(  \partial H_{g}\right)  $ induces an automorphism of
$\mathcal{C}\left(  \partial H_{g}\right)  $ which by the analogous result for
surfaces (see \cite{I3},\cite{Luo}) is geometric, that is, there exists a
homeomorphism
\[
F_{\partial H_{g}}:\partial H_{g}\rightarrow \partial H_{g}%
\]
such that $A\left(  F_{\partial H_{g}}\right)  =f|_{\mathcal{C}\left(
\partial H_{g}\right)  }.$ As $f|_{\mathcal{C}\left(  \partial H_{g}\right)
}$ maps $\mathcal{D}\left(  M\right)  $ to $\mathcal{D}\left(  M\right)  ,$
$F_{\partial H_{g}}$ sends meridian boundaries to meridian boundaries. It
follows that $F_{\partial H_{g}}$ extends to a homeomorphism $F:H_{g}%
\rightarrow H_{g}.$ We know that $A\left(  F\right)  =f$ on $\mathcal{C}%
\left(  \partial H_{g}\right)  $ and we must show that $A\left(  F\right)  =f$
on $\mathcal{I}\left(  H_{g}\right)  .$ This follows from Lemma \ref{l1} which
completes the proof that every $f\in Aut\left(  \mathcal{I}\left(
H_{g}\right)  \right)  $ is geometric. \newline Let $f\in Aut\left(
\mathcal{I}\left(  H_{g}\right)  \right)  .$ Since $A$ is shown to be onto,
there exists a homeomorphism $F:H_{g}\rightarrow H_{g}$ such that $A\left(
\left[  F\right]  \right)  =f.$ This implies that $f\left(  \mathcal{D}\left(
H_{g}\right)  \right)  =\mathcal{D}\left(  H_{g}\right)  $ and $f\left(
\mathcal{A}\left(  H_{g}\right)  \right)  =\mathcal{A}\left(  H_{g}\right)  .$
In particular, $f$ restricted to $\mathcal{C}\left(  \partial H_{g}\right)
\equiv \mathcal{D}\left(  H_{g}\right)  \cup \mathcal{A}\left(  H_{g}\right)  $
induces an automorphism $\overline{f}$ of the complex of curves $\mathcal{C}%
\left(  \partial H_{g}\right)  .$ By \cite{I3}, \cite{Luo}, there exists a
homeomorphism $F_{\partial H_{g}}:\partial H_{g}\rightarrow \partial H_{g}$
such that $A\left(  F_{\partial H_{g}}\right)  =\overline{f}.$ Such a
homeomorphism is unique unless $g=2$ in which case the map
\[
\mathcal{MCG}\left(  \partial H_{2}\right)  \rightarrow Aut\left(
\mathcal{C}\left(  \partial H_{2}\right)  \right)
\]
has a $\mathbb{Z}_{2}-$kernel generated by an involution of $\partial H_{2}.$
However, any homeomorphism of $\partial H_{g}$ which extends to $H_{g}$ it
does so uniquely (see, for example, \cite[Theorem 3.7 p.94]{F-M}), and
therefore the map
\[
\mathcal{MCG}\left(  H_{g}\right)  \rightarrow Aut\left(  \mathcal{I}\left(
H_{g}\right)  \right)
\]
is injective unless $g=2$ in which case it has a $\mathbb{Z}_{2}-$kernel.
\end{proof}

\section{Applications}

We first establish hyperbolicity for $\mathcal{I}\left(  H_{g}\right)  .$

\begin{proposition}
The complex $\mathcal{I}\left(  H_{g}\right)  $ is $\delta-$hyperbolic in the
sense of Gromov.
\end{proposition}

\begin{proof}
As far as hyperbolicity is concerned, the 1-skeleton $\mathcal{I}\left(
H_{g}\right)  ^{(1)}$ of $\mathcal{I}\left(  H_{g}\right)  $ is relevant.
$\mathcal{I}\left(  H_{g}\right)  ^{(1)}$ is endowed with the combinatorial
metric so that each edge has length $1.$ Apparently, we have an embedding
\[
i:\mathcal{C}\left(  \partial H_{g}\right)  ^{(1)}\hookrightarrow
\mathcal{I}\left(  H_{g}\right)  ^{(1)}%
\]
with $i\left(  \mathcal{C}\left(  \partial H_{g}\right)  ^{(1)}\right)
=\mathcal{D}\left(  H_{g}\right)  ^{(1)}\cup \mathcal{A}\left(  H_{g}\right)
^{(1)}$ where the superscript $^{(1)}$ always denotes $1-$skeleton. We claim
that this embedding is isometric. Indeed, if $\left[  \alpha_{1}\right]
,\left[  \alpha_{2}\right]  $ are distinct vertices with distance
$d_{\mathcal{C}}\left(  \left[  \alpha_{1}\right]  ,\left[  \alpha_{2}\right]
\right)  $ in $\mathcal{C}\left(  \partial H_{g}\right)  ^{(1)}$ then the
distance $d_{\mathcal{I}}\left(  i\left(  \left[  \alpha_{1}\right]  \right)
,i\left(  \left[  \alpha_{2}\right]  \right)  \right)  $ cannot be smaller.
For, if $\left[  S_{0}\right]  =i\left(  \left[  \alpha_{1}\right]  \right)
,\left[  S_{1}\right]  ,\ldots,\left[  S_{k}\right]  =i\left(  \left[
\alpha_{2}\right]  \right)  $ is a sequence of vertices which gives rise to a
geodesic in $\mathcal{I}\left(  M\right)  ^{(1)}$ of length less than
$d_{\mathcal{C}}\left(  \left[  \alpha_{1}\right]  ,\left[  \alpha_{2}\right]
\right)  ,$ equivalently,
\[
d_{\mathcal{I}}\left(  i\left(  \left[  \alpha_{1}\right]  \right)  ,i\left(
\left[  \alpha_{2}\right]  \right)  \right)  =k<d_{\mathcal{C}}\left(  \left[
\alpha_{1}\right]  ,\left[  \alpha_{2}\right]  \right)
\]
then for each $j=1,2,\ldots,k-1$ consider $\beta_{j}$ to be any boundary
component of $S_{j}.$ It is clear that $\beta_{j}$ is disjoint from
$\beta_{j-1}$ and $\beta_{j+1}.$ Therefore, the sequence $\left[  \alpha
_{1}\right]  ,\left[  \beta_{1}\right]  ,\ldots,\left[  \beta_{k-1}\right]
,\left[  \alpha_{2}\right]  $ is a segment in $\mathcal{C}\left(  \partial
H_{g}\right)  ^{(1)}$ of length $k$ with $k<d_{\mathcal{C}}\left(  \left[
\alpha_{1}\right]  ,\left[  \alpha_{2}\right]  \right)  $, a contradiction.

For any vertex $\left[  P\right]  $ in $\mathcal{I}\left(  H_{g}\right)
^{(1)}\setminus \mathcal{D}\left(  H_{g}\right)  ^{(1)}\cup \mathcal{A}\left(
H_{g}\right)  ^{(1)}$ we may find an annular vertex, namely, $\left[
S_{\partial P}\right]  $ where $\partial P$ is any component of the boundary
of $P,$ which is connected by an edge with $\left[  P\right]  .$ Thus,
$\mathcal{I}\left(  H_{g}\right)  ^{(1)}$ is within bounded distance from
$i\left(  \mathcal{C}\left(  \partial H_{g}\right)  ^{(1)}\right)  .$ Since
$\mathcal{C}\left(  \partial H_{g}\right)  ^{(1)}$ is $\delta-$hyperbolic in
the sense of Gromov, so is $\mathcal{I}\left(  H_{g}\right)  ^{(1)}.$
\end{proof}

An element $F\in \mathcal{MCG}\left(  H_{g}\right)  $ is called pseudo-Anosov
when it restricts to a pseudo-Anosov homeomorphism on $\partial H_{g}.$ The
proof of the following proposition is immediate from the corresponding result
for surfaces (see \cite[Prop. 4.6]{M-M2}) along with the above mentioned fact
that $\mathcal{C}\left(  \partial H_{g}\right)  $ is cobounded in
$\mathcal{I}\left(  H_{g}\right)  .$

\begin{proposition}
For any $g\geq2,$ there exists a $c>0$ such that any pseudo-Anosov
$F\in \mathcal{MCG}\left(  H_{g}\right)  ,$ any vertex $v\in \mathcal{I}\left(
H_{g}\right)  $ and any $n\in \mathbb{Z},$
\[
d_{\mathcal{I}}\left(  F^{n}\left(  v\right)  ,v\right)  \geq c\left \vert
n\right \vert .
\]

\end{proposition}

Thus, pseudo-Anosov elements in $\mathcal{MCG}\left(  H_{g}\right)  $
correspond to hyperbolic isometries of $\mathcal{I}\left(  H_{g}\right)  $ and
there are no parabolic isometries for $\mathcal{I}\left(  H_{g}\right)  .$

\end{document}